\pdfoutput=1
\documentclass[11pt]{article}

\usepackage[utf8]{inputenc}
\usepackage[T1]{fontenc}
\usepackage{lmodern}
\usepackage{amsmath,amssymb,amsthm,mathtools}
\usepackage{bbm}
\usepackage{mathrsfs}
\usepackage[margin=1in]{geometry}
\usepackage[numbers,sort&compress]{natbib}
\usepackage[colorlinks=true,linkcolor=blue,citecolor=blue,urlcolor=blue,breaklinks=true]{hyperref}

\allowdisplaybreaks

\theoremstyle{plain}
\newtheorem{theorem}{Theorem}[section]
\newtheorem{lemma}[theorem]{Lemma}
\newtheorem{corollary}[theorem]{Corollary}
\newtheorem{prop}[theorem]{Proposition}
\newtheorem{conjecture}[theorem]{Conjecture}

\theoremstyle{definition}

\newtheorem{remark}[theorem]{Remark}

\renewcommand{\deg}{\mathsf{deg}}
\renewcommand{\P}{\mathbb{P}}

\newcommand{\1}{\mathbbm{1}}
\newcommand{\G}{\mathscr{G}}
\newcommand{\C}{\mathscr{C}}

\newcommand{\E}{\mathbb{E}}

\newcommand{\R}{\mathbb{R}}
\renewcommand{\d}{\mathrm{d}}

\newcommand{\Z}{\mathbb{Z}}

\newcommand{\cG}{{{\mathbb{G}}}}

\newcommand{\bfP}{\mathbf{P}}

\newcommand{\V}{\mathscr{V}}

\newcommand{\bfE}{\mathbf{E}}

\newcommand{\PP}{\mathbf{P}}
\newcommand{\ttau}{\tau_\emptyset}
\newcommand{\scrS}{\mathscr{S}}
\newcommand{\N}{\mathbb{N}}
\providecommand{\noopsort}[1]{}

\title{Phase transitions for contact processes on sparse random graphs via metastability and local limits}
\author{
Benedikt Jahnel\thanks{Technische Universit\"at Braunschweig, Universit\"atsplatz 2, 38106 Braunschweig, Germany; Weierstrass Institute for Applied Analysis and Stochastics, Anton-Wilhelm-Amo-Str.\ 39, 10117 Berlin, Germany; \texttt{benedikt.jahnel@tu-braunschweig.de}.}
\and
Lukas L\"{u}chtrath\thanks{Weierstrass Institute for Applied Analysis and Stochastics, Anton-Wilhelm-Amo-Str.\ 39, 10117 Berlin, Germany; \texttt{lukas.luechtrath@wias-berlin.de}.}
\and
Christian M\"{o}nch\thanks{Independent Researcher, 64289 Darmstadt, Germany; \texttt{cmoench25@gmail.com}.}
}
\date{\today}

\hypersetup{
  pdftitle={Phase transitions for contact processes on sparse random graphs via metastability and local limits},
  pdfauthor={Benedikt Jahnel, Lukas Luechtrath, Christian Moench},
  pdfkeywords={local weak convergence, locality of phase transition, random network, SIS model, sparse random graph}
}

\begin{document}
\maketitle

\begin{abstract}
We propose a new perspective on the asymptotic regimes of fast and slow extinction in the contact process on locally converging sequences of sparse finite graphs. We characterise the phase boundary by the existence of a metastable density, which makes the study of the phase transition particularly amenable to local-convergence techniques. We use this approach to derive general conditions for the coincidence of the critical threshold with the survival/extinction threshold in the local limit. We further argue that the correct time scale to separate fast extinction from slow extinction in sparse graphs is, in general, the exponential scale, by showing that fast extinction may occur on stretched exponential time scales in sparse scale-free spatial networks. Together with {the results of} Nam, Nguyen and Sly (Trans.\ Am.\ Math.\ Soc.\ 375, 2022), our methods can be applied to deduce that the fast/slow threshold in sparse configuration models coincides with the survival/extinction threshold on the limiting Galton-Watson tree.
\end{abstract}

\begin{center}
\begin{minipage}{0.92\textwidth}

\end{minipage}
\end{center}

\section{Introduction and main results}

\subsection*{The contact process}
 Let $G=(V,E)$ denote a  locally finite connected graph with a dedicated root $o\in V$ and we write $(G,o)$ for the rooted graph. The \emph{contact process} with infection rate $\lambda>0$ on $G$ is the family of set-valued continuous-time Markov processes $\big\{\xi^A=(\xi^A_t)_{t\geq 0}\colon  A\subset V \text{ finite }\big\}$. The law of $\xi^A$ is determined by setting $\xi^A_0=A$ and the transition dynamics, for every $v\in V$,
 \begin{equation}\begin{aligned}
 &\xi^A_t\to\xi^A_t\setminus\{v\} \text{ at rate }1,\\
 &\xi^A_t\to\xi^A_t\cup\{v\} \text{ at rate }\lambda\sum_{w\colon v\sim w}\1\big\{w\in\xi^A_t\big\}.
 \end{aligned}
 \end{equation}
 Here and throughout we write $v\sim w$ for $\{v,w\}\in E$. For ease of notation, we set $\xi^{v}_t=\xi^{\{v\}}_t$ for $v\in V$ and analogously drop the set notation for singletons in similar instances. We write $\PP _{G}^{\lambda}$ for the law of the contact process based on the graph $G$ with infection rate $\lambda$. Three fundamental and well-known properties of the contact process are
 \begin{description}
 	\item[monotonicity:]
 		$\xi^A_t$ is stochastically increasing in $\lambda$ and $G$,
     \item[attractivity:]
     	if $A\subset B$, then $\xi^B$ stochastically dominates $\xi^A$, and
     \item[(self-)duality:]
     	$\PP_{G}^{\lambda}(\xi_t^A \cap B=\emptyset)= \PP_{G}^{\lambda}(\xi_t^B \cap A =\emptyset)$ for $A,B\subset V$ finite.
\end{description}
Duality immediately guarantees that $\xi^A$ is well-defined for any $A\subset V$.

 The \emph{extinction time} $\ttau(A)$ of $\xi^A$ denotes the time at which $(\xi^A_t)_{t\geq 0}$ first gets absorbed into the state $\emptyset$. When we work with a graph sequence, extinction and hitting times on the \(n\)-th graph are denoted by a superscript \((n)\), while unadorned times refer to fixed graphs, local limits, or auxiliary graphs specified explicitly. In this article, one central quantity of interest for the contact process on infinite graphs is the \emph{survival/extinction threshold} of $(G,o)$, given by
\[
\lambda_{\textsf{1}}(G)=\sup\big\{\lambda>0\colon  \PP _G^{\lambda}(\tau_\emptyset(o)=\infty)=0 \big\},
\]
which is guaranteed to be well-defined by monotonicity. Note that the definition of the critical rate is independent of the choice of the root, since we assume $G$ to be connected. For finite graphs, the critical rate is $0$ and contains no information about the dynamics of $\xi$, and we discuss more appropriate thresholds for this setting below. Invoking monotonicity again, one sees that, for any infinite connected graph $G$, the value $\lambda_{\textsf{1}}(G)$ is at most the extinction/survival threshold for the contact process on $\N$ with nearest-neighbour edges. The latter is well-known to be finite \cite{harris_contact_1974} and thus $\lambda_{\textsf{1}}(G)<\infty$. The survival/extinction threshold is also known as the \emph{lower critical value} for the contact process on $G$, the \emph{upper critical value} being the threshold at which the infection returns infinitely often to the root.

The contact process and related interacting particle systems were first systematically studied in this form in the 1970s, see e.g.\ \cite{harris_contact_1974, griffeath_ergodic_1975}. Classical choices for $G$ are homogeneous lattices (usually the hypercubic lattice $\mathbb{Z}^d$) and their finite subgraphs \cite{griffeath_basic_1981,griffeath_binary_1983,durrett_contact_1982,cassandro_metastable_1984,durrett_contact_1988,durrett_contact_1988b}. Using the fact that the contact process is a form of oriented percolation, Bezuidenhout and Grimmett famously proved that the critical contact process on $\Z^d$ dies out \cite{bezuidenhout_critical_1990} and established exponential decay of the volume of the infected set in \cite{bezuidenhout_exponential_1991}. A comprehensive overview of the classical theory is provided in the monographs \cite{liggett_interacting_2005,liggett_stochastic_1999}.

\subsection*{The fast/slow threshold of the contact process on random graphs.}

With the advent of network science in the early 2000's, interest in infection models on finite sparse \emph{random} graphs arose, as they serve as models for the spread of diseases or information in inhomogeneous populations. Since the contact process dies out on any finite graph, the question of extinction vs.\ survival becomes a question of fast extinction vs.\ slow extinction. Here, slow extinction means that the infection survives for a time that scales exponentially in the size of the graph. Let $(G_n)_{n\in\N}$ denote a sequence of finite graphs with $G_n=(V_n,E_n)$. Assume that $|V_n|\to\infty$, then the \emph{fast/slow threshold} for $(G_n)_{n\in\N}$ is given by
\begin{equation}\label{eq:lambda_+}
	\lambda_+\big((G_n)_{n\in\N}\big)=\sup\big\{\lambda>0\colon \lim_{n\to\infty}\mathbf{P}_{G_n}^{\lambda}\big(\tau^{(n)}_\emptyset(V_n)>{\rm e}^{c |V_n|}\big)=0\text{ for all } c>0\big\}.
\end{equation}
The main purpose of this article is to investigate the question when $\lambda_+\big((G_n)_{n\in\N}\big)$ and $\lambda_{\textsf{1}}(G)$ coincide, provided that the finite graphs $(G_n)_{n\in\N}$ approximate $(G,o)$ in a suitable sense. To formalise this, we use the framework of \cite{aldous_processes_2018,vanderhofstad2023giant,vdHGraphs2}. Let $\cG_\ast$ denote the space of equivalence classes of connected locally finite rooted graphs modulo rooted isomorphisms, equipped with the local metric $\mathsf{d}_\ast$ and let $\mathcal{P}(\cG_\ast)$ denote the space of Borel probability measures on $\cG_\ast$.
We do not usually distinguish between rooted graphs and their equivalence classes in the same vein as one commonly speaks of `a random variable $X\in L^1$'. Let us note that any locally finite rooted graph can be viewed as an element of \(\cG_\ast\) if it is identified with the connected component of its root. More background on $\cG_\ast$ as a metric space can be found in \cite{vdHGraphs2} and the references therein. Let $(\G_n)_{n\in\N}$ denote a sequence of finite connected random graphs $\G_n=(\mathscr{V}_n,\mathscr{E}_n)$ with $|\V_n|\overset{n\to\infty}{\longrightarrow}\infty$ in probability. We say $(\G_n)_{n\in\N}$ \emph{converges locally in probability} to some random rooted graph $(\G,o)$ with distribution $\mathsf{Q}\in \mathcal{P}(\cG_\ast)$ if, for any $\varepsilon>0$,
\begin{equation}\label{eq:localconv}
\lim_{n\to\infty}\P\big(\mathsf{d}_\mathcal{P}(Q_n,\mathsf{Q})>\varepsilon\big)=0,
\end{equation}
where
\[
Q_n(\, \cdot \,)=|\V_n|^{-1}\sum_{v\in\V_n}\1\{(\G_n,v)\in\,\cdot\,\}\in\mathcal{P}(\cG_\ast)
\]
denotes the empirical distribution on $\cG_\ast$ associated with $\G_n$ if a root is chosen uniformly at random, and $\mathsf{d}_\mathcal{P}$ denotes the Lévy--Prokhorov metric on $\mathcal{P}(\cG_\ast)$ {(see~\eqref{eq:metric} below for the precise definition)}.
We write
\[
    \G_n\underset{n\to\infty}{\overset{\P}{\rightharpoonup}} (\G,o)
\]
for local convergence in probability. Note that, if $o_n\in\V_n$ is chosen uniformly at random and the distribution of $(\G_n,o_n)$ is denoted by $\mathsf{Q}_n$, then \eqref{eq:localconv} implies that $\mathsf{Q}_n\to\mathsf{Q}$ weakly\footnote{Let us mention that the convergence in~\eqref{eq:localconv} still makes sense if $\mathsf{Q}$ is random, see~\cite[Remarks 2.12 and 2.13]{vdHGraphs2}. In this case, the induced weak limit is the expectation of $\mathsf{Q}$. However, since we are mostly interested in a law-of-large-numbers-type convergence of the contact process observables, we will always assume $\mathsf{Q}$ to be deterministic. Our results also hold for random $\mathsf{Q}$ in an almost-sure sense, but we do not present them in this way to spare the reader another layer of randomness that has no bearing on the technical core of our results.}.  This corresponds to \emph{local weak convergence} of $(\G_n)_{n\in\N}$ to $(\G,o)$ with distribution $\mathsf{Q}$. In particular, the above definitions apply to sequences $(G_n)_{n\in\N} $ of \emph{deterministic} finite graphs. It is well known that every distribution $\mathsf{Q}\in \mathcal{P}(\cG_\ast)$ that arises as a local limit is \emph{unimodular}, i.e., satisfies a certain mass-transport principle~\cite{aldous_processes_2018}. If a unimodular measure $\mathsf{Q}\in \mathcal{P}(\cG_\ast)$ does not admit any non-trivial representation as a convex combination of other unimodular measures, then we call $\mathsf{Q}$ \emph{extremal}. Extremal distributions on rooted graphs are characterised by the property that
\[
\mathsf{Q}(A)\in\{0,1\}, \text{ for all }A\in\mathcal{I},
\]
where the $\sigma$-field $\mathcal{I}$ consists of all Borel-events over $\mathbb{G}_\ast$ that do not depend on the root of the involved graphs.
{The extremality assumption used below plays the same role as an ergodicity assumption. Since all graphs under consideration are connected, the survival threshold $\lambda_{\mathsf 1}(G)$ is invariant under rerooting: for every $a>0$, the event $\{\lambda_{\mathsf 1}(G)\leq a\}$ belongs to $\mathcal I$. Hence, if $\mathsf Q$ is extremal, then $\lambda_{\mathsf 1}(\G)$ is $\mathsf Q$-almost surely constant; we denote this constant by $\lambda_{\mathsf 1}(\mathsf Q)$. In applications, extremality can therefore be checked by verifying triviality of the invariant $\sigma$-field, for instance through ergodicity of the underlying rooted network construction. This covers the standard local limits of ergodic random graph models, such as unimodular Galton--Watson trees and Palm versions of stationary ergodic spatial graphs.}
\begin{remark}\label{rem:giant}
In general, if a sequence $(\G_n)_{n\in \N}$ of (not necessarily connected) random graphs converges locally in probability to a limit graph $(\G,o)$ and  $(\G_n)_{n\in \N}$ is supercritical in the sense that, with high probability, a unique macroscopic \emph{giant component} is formed, then the sequence of the corresponding giant components converges locally in probability to $(\G,o)$ conditioned on $|\V|=\infty$, see \cite{vanderhofstad2023giant}. Examples where the limiting distribution is concentrated on trees include sparse Erd\"os--Renyi graphs, inhomogeneous random graphs, random regular graphs, preferential attachment graphs, and the configuration model. In the context of spatial random graphs, such as e.g.\ lattice bond and site percolation models, spatial scale-free networks, continuum percolation models and spatial random intersection graphs, the limits are usually not trees.
\end{remark}

The classical example for coincidence of the critical values $\lambda_+$ and $\lambda_{\mathsf 1}$ is the case in which $(G_n)_{n\in\N}$ are boxes in $\Z^d$ and the limit graph is the full lattice. In this case, it is also known that fast extinction actually occurs at a time scale that is logarithmic in the size of $V_n$~\cite{durrett_contact_1988,durrett_contact_1988b}. A similar result is known to hold for random regular graphs converging to the $d$-ary tree~\cite{lalley_2017}. The situation is quite different if $(\G,o)$ is not only random but also admits unbounded degrees. In two seminal works, Nam, Nguyen and Sly \cite{nam_critical_2022} and Nguyen and Sly~\cite{NguyenSly25} showed that, in configuration graphs converging to a unimodular Bienaymé--Galton--Watson tree, the extinction time in the subcritical phase is polynomial in the size of the graph and determined explicit bounds on the polynomial power. Consequently, one defines
\[
	\lambda_-\big((\G_n)_{n\in\N}\big)=\sup\big\{\lambda>0\colon \lim_{n\to\infty}\mathbf{P}^{\lambda}_{\G_n}\big(\tau^{(n)}_\emptyset(\V_n)> |\V_n|^{c}\big)=0 {\text{ for some }} c>0 \big\}.
\]
{Thus \(\lambda_-\) records extinction on some polynomial scale. Requiring the preceding display to hold for all exponents \(c>0\) would instead describe sub-polynomial extinction, which is not the convention used in Conjecture~\ref{TAMSconj}.}

The starting point of our investigation is the following conjecture of Nam, Nguyen and Sly:
\begin{conjecture}[{\cite[Conjecture 4]{nam_critical_2022}}]\label{TAMSconj}
Let $\mu$ denote a probability distribution on the non-negative integers satisfying
\[
\sum_{k\geq 0}k(k-2)\mu(k)>0
\]
and let $(\G_n)_{n\in\N}$ denote configuration graphs on $n$ vertices derived from $\mu$. Then,
\[
\lambda_-\big((\C_n)_{n\in\N}\big)=\lambda_+\big((\C_n)_{n\in\N}\big)=\lambda_\mathsf{1}(\mathscr{T}),
\]
where for each $n$, $\C_n$ is the maximal component of $\G_n$, and $\mathscr{T}$ is the local limit of $(\C_n)_{n\in\N}$, i.e., the unimodular Bienaymé--Galton--Watson tree associated with $\mu$ conditioned on non-extinction.
\end{conjecture}
For the moment, we remark that $\lambda_\mathsf{1}(\mathscr{T})>0$ if and only if $\mu$ has an exponential tail \cite{huang2018contact, bhamidi_survival_2021} and that the partial result $\lambda_+\big((\G_n)_{n\in\N}\big)\leq \lambda_\mathsf{1}(\mathscr{T})$ was already established in \cite{nam_critical_2022}.
\begin{remark}\label{rem:lambda_+}
	{Let us remark that our definition of \(\lambda_+\) describes the fast-extinction side of the exponential threshold. It is slightly more conservative than the long-survival threshold \(\lambda_c^+\) that is actually used in Nam--Nguyen--Sly~\cite{nam_critical_2022}, which is defined through high-probability survival up to \(\exp(\beta |V_n|)\) and then lets \(\beta\downarrow0\). Indeed, if \(\lambda\) belongs to the set in~\eqref{eq:lambda_+}, then survival up to \(\exp(\beta |V_n|)\) has probability tending to zero for every fixed \(\beta>0\), so \(\lambda\) cannot belong to the long-survival regime of~\cite{nam_critical_2022}. Hence, when invoking their result for configuration models, our \(\lambda_+\) is bounded above by their \(\lambda_c^+\), and the inequality \(\lambda_c^+\leq \lambda_1\) from~\cite[Theorem~5]{nam_critical_2022} applies to the threshold used here. Moreover, the conjecture implicitly claims that both definitions coincide, i.e.\ \(\lambda_+=\lambda_c^+\).}
\end{remark}

\subsection*{Overview of main results.}
{Throughout this section, we denote by} $(\G_n)_{n\in\N}$ a sequence of finite connected random graphs that converges locally in probability to {a locally finite limiting graph} $(\G,o)$. {We further assume its distribution \(Q\) to be extremal to guarantee that the threshold \(\lambda_1(Q)\) is constant.} Let $(s(m))_{m\in \N}$ be some diverging sequence. We say that the contact process on $(\G_n)_{n\in\N}$ is \emph{metastable at time scale $s(m)$} if, for all sequences $(t(n))_{n\in \N}$ with $\lim_{n\to\infty}\P (t(n)\leq s(|\V_n|))=1$, there exists some $\eta>0$, which may depend on $(\G_n)_{n\in\N}$, such that
\[
\limsup_{n\to\infty}\mathbf{P}_{\G_n}^\lambda\Big(|\V_n|^{-1}|\xi^{\V_n}_{t(n)}|>\eta\Big)>0.
\]
Our goal is to quantify the infection density $|\xi^{\V_n}|/|\V_n|$, and thereby the occurrence of a metastable phase, through the \emph{survival probability} of the contact process on the limit graph $(\G,o)$,
\[
\eta_\lambda (\mathsf{Q})=\E\big[\mathbf{P}_\G^\lambda(\tau_\emptyset(o)=\infty)\big]=\P^\lambda(\tau_\emptyset(o)=\infty).
\]
Observe that $\eta_\lambda$ depends only on the \emph{distribution} of the limit graph, while metastability as well as $\lambda_+$ and $\lambda_-$ are, in principle, dependent on the realisation of $(\G_n)_{n\in\N}$. However, if $(\G,o)$ is the infinite cluster of some percolation process, then its distribution $\mathsf Q$ is usually extremal, such as in the setting of Conjecture~\ref{TAMSconj}, and then the limiting objects do not depend on the realisation of the graph sequence, cf.~Remark~\ref{rem:giant}.

Our first main result is that the metastable density of the infection cannot exceed the survival probability in the limit.

\begin{theorem}\label{thm:metastable1}
Assume that $(\G_n)_{n\in\N}$ is a sequence of {finite} connected {random} graphs that converges locally in probability to some rooted locally finite random graph $(\G,o)$ with extremal distribution $\mathsf Q$. Then, for any diverging sequence $(t(n))_{n\in\N}$ of times and any $\varepsilon>0$, we have
\[
\lim_{n\to\infty}\P^\lambda\Big( |\V_n|^{-1}|\xi^{\V_n}_{t(n)}|\geq \eta_\lambda (\mathsf Q)+\varepsilon\Big)=0.
\]
In particular, if \(\lambda<\lambda_1(\G)\equiv\lambda_1(\mathsf Q)\), then
\[
|\V_n|^{-1}|\xi^{\V_n}_{t(n)}|\underset{n\to\infty}{\overset{\P^\lambda}{\longrightarrow}} 0.
\]
\end{theorem}
Note that the density of the infection is a global quantity, therefore the extinction probability in the limit has to coincide with the `annealed' survival probability $\eta_\lambda$.

The absence of metastability for some time scale does not imply that the contact process dies out on that time scale, as is illustrated by our next result. To formulate it, we define a sequence of random graphs $(\G_n)_{n\in\N}$ to be \emph{sparse} if the family $(\deg_{\G_n}(o_n))_{n\in\N}$ is uniformly integrable under $\P^\lambda$. Sparsity is a natural assumption in the context of locality of the fast/slow transition in the contact process as it guarantees that $(\G_n)_{n\in\N}$ is tight in $\cG_\ast$, see~\cite{BLS15}. 
\begin{theorem}\label{thm:slowextinct}
For every $\varepsilon>0$, there exists a sequence of {finite connected sparse random} graphs $(\G_n)_{n\in\N}$ that converges locally in probability to some {rooted locally finite random graph} $(\G,o)$ with extremal distribution $\mathsf Q $ satisfying $\lambda_{\mathsf{1}}(\G)>0$ and such that, for all \(\lambda>0\),
\begin{equation}\label{eq:slowextinct}
    \lim_{n\to\infty}\P^{\lambda}\Big(\tau^{(n)}_\emptyset(\V_n)>\exp\big(|\V_n|/\log^{1+\varepsilon}(|\V_n|)\big)\Big)=1.
\end{equation}
In particular, for these graph sequences we have that $0=\lambda_-\big((\G_n)_{n\in\N}\big)<\lambda_{\mathsf{1}}(\G)$.
\end{theorem}
This result suggests that, if $((\G_n,o_n))_{n\in\N}$ is sparse, then the correct time scale to distinguish the fast extinction regime from the slow extinction regime is in general the exponential one, i.e., $s(m)=\textup{e}^{cm}, m\in\N$, for some $c>0$. {Indeed, as made precise in the proof of Theorem~\ref{thm:configuration model1} below, sparsity implies that infected sets of small density have small total degree with high probability. Whenever the process visits such a low-density state, there is therefore an exponentially small, but on the exponential time scale relevant, chance that all infected vertices recover before any new infection occurs. Hence, absence of metastability on the exponential time scale implies extinction on that scale.} On the other hand, it is straightforward that, in sparse graphs, there cannot be survival at super-exponential scales. We use the connection between metastability and survival on the exponential time scale to prove a general inequality for the critical values.
\begin{theorem}\label{thm:configuration model1} Suppose $(\G_n)_{n\in\N}$ is a sequence of {finite} connected sparse {random} graphs converging locally in probability to a {rooted locally finite random} graph $(\G,o)$ with extremal distribution $\mathsf Q$. Then,
\[
    \lambda_+\big((\G_n)_{n\in\N}\big)\geq\lambda_{\mathsf{1}}(\G)\equiv \lambda_{\mathsf{1}}(\mathsf Q).
\]
\end{theorem}
In particular, our theorem implies the second equality in Conjecture~\ref{TAMSconj}.
\begin{corollary}
Let $\mu$ denote a probability distribution on the non-negative integers {satisfying \(\sum_{k\geq0}k(k-2)\mu(k)>0\)}, and let $(\G_n)_{n\in\N}$ denote configuration graphs on $n$ vertices derived from $\mu$. Then,
\[
\lambda_+\big((\C_n)_{n\in\N}\big)=\lambda_\mathsf{1}(\mathscr{T}),
\]
where, for each $n$, $\C_n$ denotes the largest component in $\G_n$ and $\mathscr{T}$ is the local limit of $(\C_n)_{n\in\N}$.
\end{corollary}
\begin{proof}
It is well known that the configuration model converges locally in probability to the unimodular Bienaymé--Galton--Watson tree $\mathscr{T}$ and that the corresponding giant components converge to the limit tree conditioned on non-extinction, see \cite{vanderhofstad2023giant}. Hence, Theorem~\ref{thm:configuration model1} applies. The converse inequality is provided in~\cite[Theorem 5]{nam_critical_2022}{, cf.\ Remark~\ref{rem:lambda_+}.}
\end{proof}
Our final result concerns lower bounds in probability for the metastable density.
For this, consider the condition
\begin{equation}\label{eq:almostlocal}
\lim_{R\to\infty}\limsup_{n\to\infty}\P^{\lambda}\Big(\xi^{o_n}_{t(n)}=\emptyset, \tau^{(n)}_R(o_n)<t(n)\Big)=0,
\end{equation}
where  $\tau^{(n)}_R(o_n)$ denotes the first time that a vertex at distance $R$ from the root is infected in $\G_n$. {The event in~\eqref{eq:almostlocal} describes precisely the possible mismatch between survival in the local limit and survival up to the finite observation time \(t(n)\): the infection started at the root has reached distance \(R\), and therefore looks locally like a surviving infection, but it has nevertheless died out before time \(t(n)\). Thus the condition says that, after first choosing a large local window, such long excursions followed by early extinction are negligible.} The next statement asserts that~\eqref{eq:almostlocal} is necessary and sufficient for lower bounding the metastable density via the limit's survival probability.
\begin{prop}\label{prop:impliesconvergence}
Let $(\G_n)_{n\in\N}$ be a sequence of {finite} connected random graphs with $\G_n\underset{n\to\infty}{\overset{\P}{\rightharpoonup}} (\G,o)$, where the limit is distributed according to some extremal distribution $\mathsf Q$. Let $(t(n))_{n\in\N}$ denote a sequence of diverging times. Then, \eqref{eq:almostlocal} is equivalent to
\[
\lim_{n\to\infty}\P^\lambda\Big( |\V_n|^{-1}|\xi^{\V_n}_{t(n)}| \leq \eta_\lambda (\mathsf Q)-\varepsilon\Big)=0,\qquad\text{ for all }\varepsilon>0.
\]
\end{prop}

This, together with our Theorem~\ref{thm:metastable1} now implies that~\eqref{eq:almostlocal} is equivalent to convergence in probability of the metastable density to the limit's survival probability.

\begin{corollary}\label{cor:metastable2}
Under the assumptions of Proposition~\ref{prop:impliesconvergence}, the condition \eqref{eq:almostlocal} is equivalent to
\[
|\V_n|^{-1}|\xi^{\V_n}_{t(n)}|\underset{n\to\infty}{\overset{\P^\lambda}{\longrightarrow}} \eta_\lambda (\mathsf Q).
\]
\end{corollary}
\begin{proof}
This is a direct consequence of Theorem~\ref{thm:metastable1} and Proposition~\ref{prop:impliesconvergence}.
\end{proof}
Naturally, establishing~\eqref{eq:almostlocal} for a given time-scale is in general hard and the main challenge in proving metastability for a concrete graph sequence.

\subsection*{Further discussion and related work.}
Another useful observation pertaining to the locality of metastability on $(\G_n)_{n\in\N}$ is that $\lambda_{\mathsf 1}(\G,o)$ can be characterised by tightness of the extinction times.
\begin{lemma}\label{lem:tightness} Suppose that $(\G_n)_{n\in\N}$ is a sequence of {finite connected random graphs} that converges locally weakly to a {rooted locally finite random graph} $(\G,o)$. For any $\lambda>0$,
 $(\tau^{(n)}_\emptyset(o_n))_{n\in\N}$ is tight if and only if
\[
    \P^\lambda(\tau_{\emptyset}(o)<\infty)=1.
\]
\end{lemma}
{The extinction time in Lemma~\ref{lem:tightness} is that of the contact process started from a single uniformly chosen vertex, since this is the local quantity dual to the density process. Indeed, by self-duality and the choice of \(o_n\) being uniformly at random,
\begin{equation}\label{eq:tight}
    \E^\lambda[\rho_n(t)]
    =\P^\lambda(o_n\in \xi_t^{\V_n})
    =\P^\lambda(\xi_t^{o_n}\neq\emptyset)
    =\P^\lambda(\tau^{(n)}_\emptyset(o_n)>t),
\end{equation}
where $\rho(t)=\rho_n(t)=|\xi^{\V_n}_{t}|/|\V_n|$ denotes the density process associated with $\xi^{\V_n}$ on $\G_n$. Thus tightness of \((\tau^{(n)}_\emptyset(o_n))_{n\in\N}\) is equivalent to a vanishing expected density, and hence to a vanishing density in probability, along every diverging sequence of times. The extinction time \(\tau^{(n)}_\emptyset(\V_n)\) is instead a global persistence time: it may be large even when the density of infected vertices is small.}
Lemma~\ref{lem:tightness} provides yet another perspective on the fast/slow transition. The law of large numbers in Corollary~\ref{cor:metastable2} suggests the following definition for the critical value
\[\lambda_\rho\big((\G_n)_{n\in\N} \big)=\sup\big\{\lambda>0\colon  \lim_{n\to\infty}\mathbf{P}_{\G_n}^\lambda\big(\rho(\textup{e}^{c|\V_n|})>\varepsilon\big)=0 \text{ for all }\varepsilon,c>0\big\}.\]
Denoting
\[
    \begin{aligned}
        \lambda_{\rho}^{-}
        &
            \big((\G_n)_{n\in\N}\big)
        \\ &
            =\sup\big\{\lambda>0\colon  \lim_{n\to\infty}\mathbf{P}_{\G_n}^\lambda\big(\rho(t(|\V_n|))>\varepsilon\big)=0 \text{ for all }\varepsilon>0 \text{ and } (t(n))_{n\in\N}\text{ with } t(n)\to\infty\big\},
    \end{aligned}
\]
we have that $\lambda^{-}_\rho\leq \lambda_\rho$. However, {by~\eqref{eq:tight} and the above}
\[
    \lim_{n\to\infty}\mathbf{P}_{\G_n}^\lambda\big(\rho(t(|\V_n|))>\varepsilon\big)=0 \quad  \text{ for all }\varepsilon>0, t(n)\to\infty,
\]
if and only if $(\tau^{(n)}_\emptyset(o_n))_{n\in\N}$ is tight. Consequently, $\lambda_\mathsf{1}=\lambda^-_{\rho}\leq \lambda_{\rho}$, in the case where the underlying graphs converge locally in probability.  Note that $ \lambda_{\mathsf 1}\leq \lambda_{\rho}$ also follows from Theorem~\ref{thm:metastable1}.
For the configuration model, the bound $\lambda_\rho\geq \lambda_{\mathsf 1}$ is essentially \cite[Theorem 6]{nam_critical_2022}. Our results show that this inequality always holds if $\G_n\to (\G,o)$ locally in probability. Moreover, our proof of Theorem~\ref{thm:configuration model1} below implies that
\[
\text{if }\G_n\underset{n\to\infty}{\overset{\P}{\rightharpoonup}} (\G,o)\text{ and }(\G_n)_{n\in\N} \text{ is sparse, then }\lambda_\rho=\lambda_+.
\]

Further evidence that the polynomial time scale is the natural one for \emph{fast} extinction in sparse locally tree-like graphs is, for instance, given in~\cite{durrett2024contactprocessesquencheddisorder} and~\cite{NguyenSly25}. Note that, unlike in our Theorem~\ref{thm:slowextinct}, the time scale of extinction in the graphs discussed there is \emph{not} determined solely by the presence of stars. Furthermore, Theorem~\ref{thm:slowextinct} should be contrasted with~\cite[Theorem~1.2]{schapira_extinction_2017}, where it is shown that the \emph{supercritical} extinction time is at least as large as the time scale given in Theorem~\ref{thm:slowextinct} on \emph{any} finite graph provided that $\lambda>\lambda_\mathsf{1}(\Z)$. Our proof also shows that the time scale $\exp(|\V_n|/\log(|\V_n|)^{1+o(1)})$ is not optimal, cf.~Remark~\ref{rem:radiuschoice}, but that $\exp(\Theta(|\V_n|))$ cannot be achieved. It is an interesting question to determine whether there are sparse graphs on which fast extinction occurs on time scale $\exp(c^-(|\V_n|))$ and slow extinction occurs on time scale $\exp(c^+(|\V_n|))$ with $c^-<c^+$, or even with $c^-=c^+$. Another recent work proving exponential extinction times with logarithmic correction for all infection rates is~\cite{barnier:hal-05064371}, see in particular Theorem 1.1.(ii) therein. However, the random graphs considered in that paper are \emph{small worlds} and it is highly likely that their local limits do not display an extinction phase in the parameter regime of~\cite[Theorem~1.1.(ii)]{barnier:hal-05064371}. This contrasts with our example in Theorem~\ref{thm:slowextinct}, which is not a small world graph but displays distances comparable to the lattice~\cite{luechtrath2024chemical, gracar2023finiteness} and does have a phase transition in the local limit.

Let us further mention that, if sparsity is violated, survival can occur on super-exponential time scales as demonstrated in~\cite{can_super-exponential_2018}.

Finally, we would like to point out that our approach to metastability in the contact process on random graphs was inspired by~\cite[Theorem~1.4]{linker_contact_2021}, a metastability result on a concrete random graph sequence, and, most prominently, by the phenomenology developed for the corresponding percolation problem~\cite{vanderhofstad2023giant}.

\subsection*{Overview of the proof section.}
The remainder of the paper is devoted to the derivation of our results. We first recall some facts about local convergence and prove Lemma~\ref{lem:tightness}, which is independent of our main results. Then we construct the example leading to Theorem~\ref{thm:slowextinct}, and finally we prove the LLN-type results Theorem~\ref{thm:metastable1} and Proposition~\ref{prop:impliesconvergence} and apply them to derive Theorem~\ref{thm:configuration model1}.

\section{Proofs}\label{sec:proofs}
\subsection{Preparatory and auxiliary results}
We recall some fundamental properties of the space $\cG_\ast$. Recall that the local metric $\mathsf{d}_\ast$ on $\cG_\ast$ is given by
\[
\mathsf{d}_\ast\big((G,o_G),(H,o_H)\big)=2^{-\sup\{k\colon  B_{G}(o_G,k)=B_{H}(o_H,k)\}},
\]
where $ B_{G}(o_G,k)$ denotes the subgraph in $G$ induced by all vertices of graph distances at most $k$ from $o_G$. Recall that statements like $B_{G}(o_G,k)=B_{H}(o_H,k)$ or $G=H$ for elements of $\cG_\ast$ always implicitly refer to equality of equivalence classes under rooted isomorphisms.

We begin by noting that our definition of local convergence in probability is slightly different to the one given in \cite{vanderhofstad2023giant,vdHGraphs2}. In fact, there are several equivalent formulations.
{Recall that the Levy--Prokhorov metric on \(\mathcal P(\cG_\ast)\) is defined, for \(\mu,\nu\in\mathcal P(\cG_\ast)\), by
\begin{equation}\label{eq:metric}
    \mathsf d_\mathcal P(\mu,\nu)=\inf\big\{\varepsilon>0\colon
    \mu(A)\leq \nu(A^\varepsilon)+\varepsilon\text{ and }
    \nu(A)\leq \mu(A^\varepsilon)+\varepsilon
    \text{ for all Borel }A\subset\cG_\ast\big\},
\end{equation}
where
\[
    A^\varepsilon=\big\{(G,o)\in\cG_\ast\colon
    \inf_{(H,o_H)\in A}\mathsf d_\ast\big((G,o),(H,o_H)\big)<\varepsilon
    \big\}.
\]
It metrises weak convergence of probability measures on \(\cG_\ast\); see, for instance, \cite[Chapter~2]{vdHGraphs2}. The following equivalences are standard characterisations of local convergence in probability.}
\begin{lemma}\label{lem:localconvergence}
Let $(\G_n)_{n\in\N}$ be a sequence of finite random graphs and let $(\G,o)\in\cG_\ast$ be a random graph with distribution $\mathsf{Q}$. The following three assertions are equivalent:
\begin{enumerate}
    \item[(a)] $\G_n\underset{n\to\infty}{\overset{\P}{\rightharpoonup}} (\G,o)$
    \item[(b)] For all bounded continuous functions $h\colon \cG_\ast\to\R$,
\[
\E[h(\G_n,o_n)|\G_n] \underset{n\to\infty}{\overset{\P}{\longrightarrow}} \int h(G,o)\,\d\mathsf{Q}(G,o),
\]
where $o_n\in\V_n$ is chosen uniformly at random and $\underset{n\to\infty}{\overset{\P}{\longrightarrow}}$ denotes convergence in probability in $\R$.\footnote{Recall that local convergence in probability in particular implies that $\E[h(\G_n,o_n)] \to\int h(G,o)\,\d\mathsf{Q}(G,o)$, as \(n\to\infty\), for all bounded continuous functions $h$. This characterises local weak convergence by the Portmanteau theorem. Further, since the limit on the right-hand side is deterministic,  local convergence in probability is equivalent to convergence in distribution of all random variables $\E[h(\G_n,o_n)|\G_n]$ towards $\int h(G,o)\,\d\mathsf{Q}(G,o)$. On the other hand, convergence in probability on $\cG_\ast$ with respect to the local topology induced by $\mathsf{d}_\ast$ is stronger than local convergence in probability.}
\item For all bounded continuous functions $h_1,h_2\colon \cG_\ast\to\R$,
\[
\E[h_1(\G_n,o_n)h_2(\G_n,\tilde o_n)] \underset{n\to\infty}{{\longrightarrow}} \int h_1(G,o)\,\d\mathsf{Q}(G,o)\int h_2(G,o)\,\d\mathsf{Q}(G,o),
\]
where $o_n,\tilde o_n\in \V_n$ represent two independently and uniformly chosen roots.
\end{enumerate}
\end{lemma}

\begin{proof}
The equivalence of (a) and (b) is precisely the bounded-continuous-test-function characterisation of local convergence in probability; see~\cite[Theorem~2.15(b)]{vdHGraphs2}. Since the limiting measure \(\mathsf Q\) is deterministic here, convergence in distribution of the conditional expectations appearing there is equivalent to convergence in probability. The equivalence of (b) and (c) is presented, for example, in~\cite[Lemma~2.8]{Lacker23}.
\end{proof}

We next explain how the contact process can be incorporated into the local convergence setup. The notion of \emph{random networks} describing processes on random graphs that are invariant under graph isomorphisms was coined in \cite{aldous_processes_2018}. A \emph{network} is a rooted graph $(G,o)$ together with two maps $\Xi_V\colon V\to \mathcal{S},\,  \Xi_E\colon E\to \mathcal{S}$ on edges and vertices. Here, $\mathcal{S}$ is some metric mark space. We choose $\mathcal{S}=\mathcal{N}(\R\times\{1,-1\})$, the space of all locally finite point measures on $\R$ marked by $-1$ or $1$. We turn the network into a random network by allowing the underlying graph $\G$ to be random, and conditionally on $\G=G=(V,E)$, letting $(\Xi_E(e))_{e\in E}$ be an independent collection of $\operatorname{Poisson}(2\lambda)$-processes where each point carries mark $1$ or $-1$ with probability $1/2$ and by letting $(\Xi_V(v))_{v\in V}$ be an independent collection of $\operatorname{Poisson}(1)$-processes where each point carries mark $1$. A realisation of the contact process on $G$ with initial infections at $A\subset V$ is now obtained by considering infection paths induced by interpreting $(\Xi_E(e))_{e\in E}$ as infection events with orientations along edges and $(\Xi_V(v))_{v\in V}$ as recovery events on the vertices. This is the well-known \emph{graphical representation} of the contact process. It is straightforward to see that the law $\mathbf{P}_G^\lambda$ of the induced oriented percolation model is invariant under rooted isomorphisms of $(G,o)$. If $(G,o)$ is selected by some random mechanism, we are therefore justified in interpreting the annealed law $\P^\lambda$ as a distribution on random networks in the sense of \cite{aldous_processes_2018}, in particular we may write
\[
\P^\lambda =\int \mathbf{P}_{G}^\lambda \;\d \mathsf{Q}(G,o)
\]
to designate the law of the contact process on a random graph with distribution $\mathsf{Q}\in \mathcal{P}(\cG_\ast)$. A similar definition applies to sequences of random graphs and henceforth we will assume, without loss of generality, that the graph sequence, limit graph, and all associated contact processes live on the same probability space and have distribution $\P^\lambda$. This is only for notational convenience, since the framework of local convergence in probability is flexible enough to deal with graph sequences that do not converge on the same probability space, see \cite[Remark 2.12]{vdHGraphs2} for a discussion. We occasionally omit the parameter $\lambda$ from the notation if we refer to distributional properties of the underlying graphs under $\P^\lambda$ only.

We close this section with the proof of Lemma~\ref{lem:tightness}, which is independent of our main results.
\begin{proof}[Proof of Lemma~\ref{lem:tightness}]
{Write \(T_n=\tau^{(n)}_\emptyset(o_n)\) and \(T=\tau_\emptyset(o)\). For \(R\in\N\), let \(\tau_R^{(n)}\) and \(\tau_R\) denote the first times at which the infection started at the root reaches a vertex at graph distance \(R\) from the root in \(\G_n\) and \(\G\), respectively. Then, for fixed \(s\geq0\) and \(R\in\N\),
\begin{equation}\label{eq:absvalbound}
	\begin{aligned}
	\big|\P^\lambda
	&
		(T_n>s)-\P^\lambda(T>s)\big|
	\\ &
		\leq \big|\P^\lambda(T_n>s,\tau_R^{(n)}>s)-\P^\lambda(T>s,\tau_R>s)\big| \;+\;\P^\lambda(\tau_R^{(n)}\leq s)+\P^\lambda(\tau_R\leq s).
	\end{aligned}
\end{equation}
For fixed \(R\) and \(s\), the event \(\{T>s,\tau_R>s\}\), as well as the event \(\{\tau_R\leq s\}\), is determined by the marked \(R\)-neighbourhood of the root in the graphical representation, up to time \(s\). Therefore, by local weak convergence of the marked networks,
\[
    \P^\lambda(T_n>s,\tau_R^{(n)}>s)\to \P^\lambda(T>s,\tau_R>s),
    \qquad\text{and}\qquad
    \P^\lambda(\tau_R^{(n)}\leq s)\to \P^\lambda(\tau_R\leq s).
\]
Taking the limsup in~\eqref{eq:absvalbound} gives
\[
    \limsup_{n\to\infty}\big|\P^\lambda(T_n>s)-\P^\lambda(T>s)\big|
    \leq 2\P^\lambda(\tau_R\leq s).
\]
Letting \(R\to\infty\), the right-hand side vanishes by non-explosion of the graphical representation. Hence \(T_n\) converges in distribution to \(T\), in the sense that \(\P^\lambda(T_n>s)\to\P^\lambda(T>s)\) for every \(s\geq0\).

If the contact process on \(\G\) dies out almost surely, then \(T<\infty\) almost surely. The convergence of the tails just established implies tightness of \((T_n)_{n\in\N}\). Conversely, if the contact process on \(\G\) survives with probability \(p>0\), then \(\P^\lambda(T>s)\geq p\) for every \(s\geq0\). Tightness of \((T_n)_{n\in\N}\) would allow us to choose \(s\) with \(\sup_n\P^\lambda(T_n>s)<p/2\), contradicting \(\P^\lambda(T_n>s)\to\P^\lambda(T>s)\geq p\).}
\end{proof}


\subsection{Slow extinction of subcritical contact processes on scale-free spatial networks}
In this section, we prove Theorem~\ref{thm:slowextinct}. To this end, we need to construct, for any given \(\varepsilon>0\), a sequence of graphs $(\G_n)_{n\in \N}$ that satisfies~\eqref{eq:slowextinct}. We first introduce an auxiliary graph. Place \(n\) independent random variables uniformly on the one-dimensional torus \((-n/2,n/2]\) of volume \(n\), order them from smallest value to largest (with respect to the interval) and denote by \(X_{-n/2+1}<\dots<X_{n/2}\) the ordered sequence; we assume \(n\) to be even for notational convenience. We now assign each vertex \(X_i\) an independent radius \(R_i\), drawn as an i.i.d.\ copy of a positive random variable \(R\), and form the associated Boolean graph by connecting \(X_i\) and \(X_j\) precisely if \(\d_n(X_i,X_j)<R_i+R_j\), where \(\d_n\) denotes the torus metric. We denote the resulting graph as \(\widetilde{\G}_n\). The graph \(\G_n\) now has the discrete torus \(\{-n/2+1,\dots,n/2\}\) as its vertex set and
\[
    \mathscr{E}_n :=\big\{i\sim j\colon X_i\sim X_j \text{ in }\widetilde{\G}_n\big\} \cup \big\{i\sim i+1\colon i=-n/2+1,\dots,n/2-1\big\}\cup\big\{n/2\sim 1-n/2\big\}
\]
as its edge set, writing `$\sim$' to indicate neighbours. {We construct the limiting graph \(\G\) by the same rule, but with the vertex set of the auxiliary graph replaced by the Palm version of a Poisson process on the real line; this is the \emph{augmented Boolean model} analysed in~\cite{jahnel2025phasetransitionscontactprocesses}. We identify the origin \(0\) in \(\Z\) as the root of \(\G\).

The local convergence of \((\G_n)_{n\in\N}\) to \((\G,0)\) follows from the usual Palm convergence of binomial point processes on growing tori to the Poisson process on \(\R\); see, for instance,~\cite{last_lectures_2018}. Indeed, if \(o_n\) is chosen uniformly from \(\{-n/2+1,\dots,n/2\}\), and if the torus is lifted to the interval centred at \(X_{o_n}\), then the cyclically relabelled marked point process
\[
    \sum_i\delta_{(X_{o_n+i}-X_{o_n},R_{o_n+i})}
\]
converges on bounded windows to the Palm version of the marked Poisson process, i.e., to a Poisson process with a point at the origin carrying an independent radius. The added nearest-neighbour cycle becomes the nearest-neighbour graph on \(\Z\). For the finite-first-moment radius laws used below, the limiting augmented Boolean graph is locally finite and every fixed rooted graph ball is almost surely determined by a sufficiently large finite window of this marked process; moreover, ties of the form \(|x-y|=R_x+R_y\) have probability zero. Hence \(B_{\G_n}(o_n,r)\) converges in distribution to \(B_\G(0,r)\) for every \(r\). Applying the same finite-window argument to empirical averages of bounded local test functions, together with the spatial law of large numbers for the binomial process on the torus, gives convergence of the empirical rooted-graph measures in probability.

The distribution of \((\G,0)\) is extremal. To see this, enumerate the Palm Poisson points increasingly as \((X_i)_{i\in\Z}\), with \(X_0=0\), and write \(S_i=X_{i+1}-X_i\). The two-sided marked sequence \((S_i,R_i)_{i\in\Z}\) is stationary and ergodic under index shifts, and the rooted graph \((\G,0)\) is a factor of this sequence. Rerooting \(\G\) at another vertex corresponds to such an index shift. Hence every rerooting-invariant event is contained, up to null sets, in the shift-invariant \(\sigma\)-field of the underlying marked sequence, which is trivial. This is precisely extremality.

The same finite-first-moment condition gives sparsity of the finite graph sequence. Let \(D_n^{\rm B}\) be the degree contribution of the Boolean edges at a uniformly chosen root. Conditionally on two radii, a second vertex is connected to the root with probability at most \(\min\{2(R+R')/n,1\}\), and therefore \(\sup_n\E[D_n^{\rm B}]\leq 4\E[R]\). To obtain uniform integrability, truncate the radii at a level \(L\). Edges whose two incident radii are at most \(L\) are contained in the set of points within torus distance \(2L\) of the root; this number has uniformly bounded second moments for fixed \(L\). The expected number of incident Boolean edges for which at least one of the two radii exceeds \(L\) is bounded, uniformly in \(n\), by a constant times \(\E[R\1\{R>L\}]+\P(R>L)\). Sending \(L\to\infty\) shows that \((D_n^{\rm B})_{n\in\N}\) is uniformly integrable. Since the added cycle contributes at most two further neighbours, \((\G_n)_{n\in\N}\) is sparse whenever \(\E[R]<\infty\).

Finally, positivity of the limiting critical value is exactly the augmented-Gilbert-graph case of~\cite[Theorem~2.5]{jahnel2025phasetransitionscontactprocesses}. The assumptions of that theorem are the stationarity and ergodicity verified above, the addition of the nearest-neighbour graph on \(\Z\), and integrability of the radius law, \(\E[R]<\infty\). Hence \(\lambda_1(\G)>0\).

The proof of Theorem~\ref{thm:slowextinct} then exploits a finite-graph effect that is invisible in the local limit. We choose below an integrable radius law with tail just heavy enough that the maximum radius is, with high probability, at least \(K n/\log^{1+\varepsilon}(n)\) for every fixed \(K\). The vertex carrying this radius has, by binomial concentration, degree of the same order. The star induced by this vertex and its neighbours survives for time exponential in its degree, so after choosing \(K>1/c_\lambda\), where \(c_\lambda\) is the star-survival constant in~\eqref{eq:starTailBound}, this single star forces survival of the contact process on \(\G_n\) beyond \(\exp(n/\log^{1+\varepsilon}(n))\) with high probability.}

\begin{proof}[Proof of Theorem~\ref{thm:slowextinct}]
The proof is based on a comparison of the survival time on the graph \(\G_n\) with the time the process survives on the star graph \(\scrS_n^*\), induced by the vertex of maximal degree. To this end, let us first recapitulate some known results on the survival time on star graphs. Let \(\scrS^{(k)}\) be the star graph with \(k\) leaves and centre \(o\). The first result about the survival time on \(\scrS^{(k)}\) is \cite[Lemma~5.2]{berger_spread_2005}, stating that the infection survives, with high probability, exponentially long in the number of leaves. We use a refined version of this statement. First, by~\cite[Lemma~2.5]{schapira_extinction_2017}, we have
\[
\bfE_{\scrS^{(k)}}^\lambda\big[\tau_\emptyset(\scrS^{(k)})\big]\geq {\rm e}^{2 c_\lambda k},
\]
where \(c_\lambda>0\) is a \(\lambda\)-dependent constant. Secondly, we have, due to \cite[Lemma~2.13]{valesin_survey}, that
\[
	\bfP_{\scrS^{(k)}}^\lambda \big(\tau_\emptyset(\scrS^{(k)})\leq t\big)\leq \frac{t}{\bfE^\lambda_{\scrS^{(k)}}\big[\tau_\emptyset(\scrS^{(k)})\big]}.
\]
Combining both results and choosing \(t={\rm e}^{c_\lambda k}\), we thus infer
\begin{equation}\label{eq:starTailBound}
	\bfP_{\scrS^{(k)}}^\lambda\big(\tau_\emptyset(\scrS^{(k)})>{\rm e}^{c_\lambda k}\big)\geq 1-{\rm e}^{-c_\lambda k}.
\end{equation}
Now, fix \(\varepsilon>0\), choose \(1<p<1+\varepsilon\), and consider the augmented Boolean graph \(\G_n\), constructed above, with radius distribution
  \begin{equation}\label{eq:radiuschoice}
  		\P(R>x)=\big(x\log^p(x)\big)^{-1}.
  \end{equation}
Note that \(R\) has finite first moment as \(p>1\) {so that the local limit \((\G,0)\) is sparse and satisfies \(\lambda_1(\G)>0\)}. By independence of the radii, we infer for the largest radius \(R^*_n:=\max\{R_{1-n/2},\dots,R_{n/2}\}\), any \(K>0\) and large enough \(n\) that
\begin{equation*}
	\P\big(R^*_n\leq \tfrac{K n}{\log^{1+\varepsilon}(n)}\big) = \big(1-\tfrac{\log^{1+\varepsilon}(n)}{K n\log^p(K n/\log^{1+\varepsilon}(n))}\big)^n \leq \mathrm{e}^{-\log^{1+\varepsilon-p}(n)/K},
\end{equation*}
which tends to zero, as \(p<1+\varepsilon\). Put differently, \(R^*_n> Kn/\log^{1+\varepsilon}(n)\), for any \(K\), with high probability. Moreover, by construction, the degree of \(X^*_n\), the vertex associated with \(R^*_n\), has degree lower bounded by a Binomial number with parameters \(n\) and \(2R^*_n/n\). Writing \(X_{n,p}\) for a binomial with parameters \(n\) and \(p\), Chernoff's inequality thus yields, for sufficiently large \(n\),
\[
    \begin{aligned}
        \P\big(\deg(X^*_n)> \tfrac{K n}{\log^{1+\varepsilon}(n)}\big)
        &
           = \E\big[\P\big(\deg(X^*_n)> \tfrac{K n}{\log^{1+\varepsilon}(n)} \, \big| \, R_n^*\big)\big]
         \\ &
         	\geq \E\big[\P\big(X_{n, 2K/\log^{1+\varepsilon}(n)}> \tfrac{K n}{\log^{1+\varepsilon}(n)}\big)\1\big\{R^*_n> \tfrac{K n}{\log^{1+\varepsilon}(n)}\big\}\big]
        \\ &
            \geq \big(1-\mathrm{e}^{-Kn/(4\log^{1+\varepsilon}(n))}\big) \P\big(R^*_n>\tfrac{K n}{\log^{1+\varepsilon}(n)}\big)
         \\ &
            \geq \Big(1-\mathrm{e}^{-Kn/(4\log^{1+\varepsilon}(n))}\Big)\big(1-\mathrm{e}^{-\log^{1+\varepsilon-p}(n)/K}\big).
    \end{aligned}
\]
Combining this with~\eqref{eq:starTailBound}, and writing \(\scrS^*_n\) for the star graph induced by \(X^*_n\) and its neighbours, we have for \(K=1/c_\lambda\),
\[
	\begin{aligned}
		\P^\lambda \big(\tau^{(n)}_\emptyset(\V_n)>{\rm e}^{n/\log^{1+\varepsilon}(n)}\big)
		&
			\geq \P^\lambda \big(\tau_\emptyset(\scrS^*_n)>{\rm e}^{n/\log^{1+\varepsilon}(n)}\big)
		\\ &
			\geq \E\Big[\1\big\{\deg(X_n^*)\geq \tfrac{n}{c_\lambda\log^{1+\varepsilon}(n)}\big\}\bfP^\lambda_{\scrS^*_n}\big(\tau_\emptyset(\scrS_n^*)\geq \mathrm{e}^{n/\log^{1+\varepsilon}(n)}\big)\Big]
		\\ &
			\geq \big(1-\mathrm{e}^{-\tfrac{n}{4c_\lambda\log^{1+\varepsilon}(n)}}\big)\big(1-\mathrm{e}^{-c_\lambda\log^{1+\varepsilon-p}(n)}\big)\big(1-\mathrm{e}^{-n/\log^{1+\varepsilon}(n)}\big)\underset{n\to\infty}\longrightarrow 1,
	\end{aligned}
\]
as desired. In particular, this holds true for any \(\lambda\in(0,\lambda_1(\G))\) and hence \(\lambda_1(\G)>\lambda_-\big((\G_n)_{n\in\N}\big)=0\).
\end{proof}
\begin{remark}\label{rem:radiuschoice}
The choice of the distribution in~\eqref{eq:radiuschoice} can be adapted to yield examples in which the survival time is even closer to $\textup{e}^{\Theta(|\V_n|)}$, for instance, by setting
\[
    \P(R>x)\sim\big(x\log (x) \log^2(\log(x))\big)^{-1}.
\]
In particular, the resulting graph would satisfy the lower bound in Theorem~\ref{thm:slowextinct} for all $\varepsilon$ simultaneously. We have chosen the particular scaling in \eqref{eq:radiuschoice} to illustrate that we can match the best known \emph{universal} lower bound on supercritical extinction times given in \cite{schapira_extinction_2017} with subcritical extinction times on certain classes of graphs.
\end{remark}

\subsection{Upper bounds on the metastable density from local limit}
Before we provide the proof of Theorem~\ref{thm:metastable1}, we introduce some notation and auxiliary results. Throughout this section, let $\big((\G_n,o_n)\big)_{n\in\N}$ and its local limit $(\G,o)$ with distribution $\mathsf Q$ be given. As usual, \(o_n\) denotes a uniformly chosen vertex here. Fix $\lambda>0$, set $\eta=\eta_\lambda(\mathsf Q)$ and let $\tau_R^{(n)}(v)$ denote the first time that the infection started in $v\in \V_n$ reaches a vertex of graph distance $R$ from $v$ in $\G_n$. For \(v\in \V\), let \(\tau_R(v)\) denote the corresponding quantity for the rooted limit graph $(\G,o)$. Now define
\[
{\eta_R}=\E[\mathbf{P}_{\G}^\lambda(\tau_R(o)<\infty)],\quad R\in\N,
\]
and
\[
{Z_R}=Z_R(n)=\sum_{v\in \V_n}\1\big\{\tau_R^{(n)}(v)<\infty \big\}, \quad n\in\N.
\]
{The proof of Theorem~\ref{thm:metastable1} compares survival up to the observation time \(t(n)\) with the local event that a single infection reaches distance \(R\) from its starting point. The first auxiliary statement says that an infection which survives until a diverging time but has not reached distance \(R\) is negligible: for fixed \(R\), the event is governed by the process confined to the rooted \(R\)-neighbourhood, and local convergence reduces the estimate to finitely many finite neighbourhoods with high probability. The second auxiliary statement says that the empirical fraction of vertices from which distance \(R\) is ever reached converges to the corresponding probability in the local limit. Extremality is used here to make this limiting fraction deterministic. We state these two estimates next.}

\begin{prop}\label{prop:noboundedsurvival}
Assume that $(\G_n)_{n\in\N}$ is a sequence of {finite connected random graphs} with $\G_n\underset{n\to\infty}{\overset{\P}{\rightharpoonup}} (\G,o)$. Then,
\[
|\V_n|^{-1}\sum_{v\in \V_n}\1\big\{\xi^{v}_{t(n)}\neq \emptyset, \tau_R^{(n)}(v)>t(n)\big\}\underset{n\to\infty}{\overset{\P^\lambda}{\longrightarrow}}0,
\]
for any $R\in\N$ and any diverging sequence $(t(n))_{n\in\N}$.
\end{prop}
\begin{prop}\label{prop:conversurvpath}
Assume that $(\G_n)_{n\in\N}$ is a sequence of {finite connected random graphs} with $\G_n\underset{n\to\infty}{\overset{\P}{\rightharpoonup}} (\G,o)$ where $(\G,o)$ is distributed according to some extremal measure $\mathsf Q$. Then we have, for any $R\in\N$, that \[|\V_n|^{-1}Z_R\underset{n\to\infty}{\overset{\P^\lambda}{\longrightarrow}}\eta_R.\]
\end{prop}
Before we prove Propositions~\ref{prop:noboundedsurvival} and~\ref{prop:conversurvpath}, we show how they imply Theorem~\ref{thm:metastable1}.
\begin{proof}[Proof of Theorem~\ref{thm:metastable1}]
By self-duality, we have that
\[
\P^\lambda\big( \big|\xi_{t(n)}^{\V_n}\big|\in \cdot \big) = \P^\lambda\Big( \sum_{v\in \V_n}\1\big\{\xi_{t(n)}^v\neq \emptyset\big\}\, \in \cdot \Big)
\]
and hence the assertion of the theorem is equivalent to showing
\begin{equation}\label{eq:toshow}
\Big(|\V_n|^{-1}\sum_{v\in \V_n}\1\big\{\xi_{t(n)}^v\neq \emptyset\big\}- \eta\Big)\vee 0 \underset{n\to\infty}{\overset{\P^\lambda}{\longrightarrow}} 0.
\end{equation}
Observe that, on one hand,
\[
\eta=\mathbb{P}^\lambda\Big(\bigcap_{R\in\N}\{\tau_R(o)<\infty\} \Big)=\lim_{R\to\infty}\eta_R,
\]
and hence, for any given $\varepsilon$, we may chose $R_{\varepsilon}$ so large that, for any $R>R_{\varepsilon}$,
\begin{equation}\label{eq:eta_Rbound}
|\V_n|^{-1}\sum_{v\in \V_n}\1\big\{\xi_{t(n)}^v\neq \emptyset\big\}\leq \eta_R+\varepsilon/2
\end{equation}
implies
\[
|\V_n|^{-1}\sum_{v\in \V_n}\1\big\{\xi_{t(n)}^v\neq \emptyset\big\}\leq \eta+\varepsilon.
\]
On the other hand, the bound
\[
	\begin{aligned}
		|\V_n|^{-1}
		&
			\sum_{v\in \V_n}\1\big\{\xi_{t(n)}^v\neq \emptyset\big\}
		\\ &
			= |\V_n|^{-1}\sum_{v\in \V_n}\1\big\{\xi_{t(n)}^v\neq \emptyset, \tau_R^{(n)}(v)\leq t(n)\big\}+|\V_n|^{-1}\sum_{v\in \V_n}\1\big\{\xi_{t(n)}^v\neq \emptyset, \tau_R^{(n)}(v)> t(n)\big\}\\
		& \leq |\V_n|^{-1}\sum_{v\in \V_n}\1\big\{\tau_R^{(n)}(v)< \infty\big\}+|\V_n|^{-1}\sum_{v\in \V_n}\1\big\{\xi_{t(n)}^v\neq \emptyset, \tau_R^{(n)}(v)> t(n)\big\},
	\end{aligned}
\]
together with Proposition~\ref{prop:noboundedsurvival} tells us that
\[
\Big(|\V_n|^{-1}\sum_{v\in \V_n}\1\big\{\xi_{t(n)}^v\neq \emptyset\big\}-|\V_n|^{-1}Z_R\Big)\vee 0 \underset{n\to\infty}{\overset{\P^\lambda}{\longrightarrow}} 0.
\]
for any $R\in\N$. Choosing $R>R_\varepsilon$ and applying Proposition~\ref{prop:conversurvpath} thus yields that \eqref{eq:eta_Rbound} occurs with probability tending to $1$ as $n\to\infty$, which in turn establishes \eqref{eq:toshow} and concludes the proof.
\end{proof}
It remains to prove the two supporting results.
\begin{proof}[Proof of Proposition~\ref{prop:noboundedsurvival}]
For all $\varepsilon>0$, we have
\[\begin{aligned}
    \P^\lambda\Big(|\V_n|^{-1}
    &
        \sum_{v\in \V_n}\1\{\xi^{v}_{t(n)}\neq \emptyset, \tau_R^{(n)}(v)>t(n)\}>\varepsilon\Big)
    \\&
        \le \frac{1}{\varepsilon}\E^\lambda\Big[|\V_n|^{-1}\sum_{v\in \V_n}\1\{\xi^{v}_{t(n)}\neq \emptyset, \tau_R^{(n)}(v)>t(n)\}\Big]
    \\&
        \le \frac{1}{\varepsilon} \E^\lambda \Big[|\V_n|^{-1}\sum_{v\in \V_n}f_{t(n)}\big(B_{\G_n}(v,R)\big)\Big]
    \\&
        =\frac{1}{\varepsilon}\sum_{(H,o_H)\in \tilde\cG_\ast}f_{t(n)}(H,o_H) \P\big(B_{\G_n}(o_n,R)= (H,o_H)\big),
\end{aligned}\]
where
\[
    f_{t(n)}(G,o):=\PP ^\lambda_G\big(\xi^{o}_s\neq \emptyset \text{ for all } s\le t(n)\big), \quad (G,o)\in\cG_\ast.
\]
By the assumption of local convergence in probability, the distributions of $B_{\G_n}(o_n,R)$ converge to the distribution of $B_{\G}(o,R)$. Since the limiting graph is locally finite, $\P(B_{\G}(o,R)\in \; \cdot\;)$ is a probability measure on $\cG_\ast$ and hence, by Prokhorov's theorem, the distributions of $\big(B_{\G_n}(o_n,R)\big)_{n\ge 1}$ are tight. Hence, for all $\delta>0$, there exists a finite set $\mathcal A\subset \tilde\cG_\ast$ such that
\[
    \sup_{n\ge 1}\sum_{(H,o_H)\notin \mathcal A}\P^\lambda\big(B_{\G_n}(o_n,R)= (H,o_H) \big)\le \delta.
\]
We thus conclude that
\[\begin{aligned}
& \limsup_{n\to\infty}\P^\lambda\Big(|\V_n|^{-1}\sum_{v\in \V_n}\1\{\xi^{v}_{t(n)} \neq \emptyset, \tau_R^{(n)}(v)>t(n)\}>\varepsilon\Big)\\
&\le \limsup_{n\to\infty}\sum_{(H,o_H)\in \mathcal A} f_{t(n)}(H,o_H)\P\big(B_{\G_n}(o_n,R)= (H,o_H) \big)+ \frac{\delta}{\varepsilon}= \frac{\delta}{\varepsilon},
\end{aligned}
\]
since $\lim_{n\to\infty}\PP _H^\lambda\big(\xi^{o_H}_s\neq \emptyset \text{ for all } s\le t(n)\big)=0$ for all finite $H$. This gives the result.
\end{proof}

\begin{proof}[Proof of Proposition~\ref{prop:conversurvpath}]
We have that
\[
\lim_{n\to\infty}\E^\lambda\big[|\V_n|^{-1}Z_R-\eta_R \big]= \lim_{n\to\infty}\E\big[\mathbf{P}_{\G_n}^\lambda\big(\tau_R^{(n)}(o_n)<\infty\big)\big|\G_n\big]-\eta_R=0,
\]
by Lemma~\ref{lem:localconvergence}(b), hence the first moments asymptotically agree under $\P^\lambda$. On the other hand,
\[
\E^{\lambda}\big[|\V_n|^{-2}Z_R^2\big|\G_n\big] =\E\big[ \mathbf{P}_{\G_n}^\lambda\big(\tau_R^{(n)}(o_n)<\infty,\tau_R^{(n)}(o'_n)<\infty\big) \big|\G_n\big],
\]
where $(o_n,o'_n)$ is uniformly chosen (with replacement) from $\V_n\times \V_n$. Since the events $\{\tau_R^{(n)}(o_n)<\infty\}$ and $\{\tau_R^{(n)}(o'_n)<\infty\}$ are measurable with respect to the marks \(\Xi_V,\Xi_E\) of the random network and the graph inside a radius of $R+1$ around the respective root, the right-hand side converges to $\eta_R^2$ as $n\to\infty$ by Lemma~\ref{lem:localconvergence}(c). Combining the first and second moment limits yields that the conditional variance of ${Z_R}/{|\V_n|}$, given $\G_n$, vanishes, implying that ${Z_R}/{|\V_n|}$ converges to its expectation $\eta_R$ in probability and in $L1$.
\end{proof}

\subsection{Lower bounds on the metastable density from local limits}
In this section, we prove Proposition~\ref{prop:impliesconvergence}.
\begin{proof}[Proof of Proposition~\ref{prop:impliesconvergence}]
Assume that $\G_n\underset{n\to\infty}{\overset{\P}{\rightharpoonup}} (\G,o)$ and let $(t(n))_{n\in\N}$ denote a sequence of diverging times. We begin by showing that
\[
|\V_n|^{-1}\sum_{v\in \V_n}\1\big\{\xi_{t(n)}^v\neq \emptyset\big\}\underset{n\to\infty}{\overset{\P^\lambda}{\longrightarrow}} \eta_\lambda(\mathsf Q)
\]
implies \[\lim_{R\to\infty}\limsup_{n\to\infty}\P^{\lambda}\big(\xi^{o_n}_{t(n)}=\emptyset, \tau^{(n)}_R(o_n)<t(n)\big)=0.\]
To this end, observe that, for all $\varepsilon>0$,
\begin{equation}\label{eq:convEBound}
    \begin{aligned}
        \E^\lambda\Big[
        &
        	|\V_n|^{-1} \sum_{v\in \V_n}\1\big\{\xi_{t(n)}^v= \emptyset, \tau^{(n)}_R(v)\leq t(n)\big\}\Big]
        \\ &
        	\le \varepsilon+\P^\lambda\Big( \sum_{v\in \V_n}\1\big\{\xi_{t(n)}^v= \emptyset, \tau^{(n)}_R(v)\leq t(n)\big\}>|\V_n|\varepsilon\Big).
    \end{aligned}
\end{equation}
The random variable of the second term can be rewritten as
\[
    \begin{aligned}
        \sum_{v\in\V_n}
        &
        	\1\big\{\xi_{t(n)}^v= \emptyset, \tau^{(n)}_R(v)\leq t(n)\big\}
        \\ &
            \leq \sum_{v\in\V_n}\1\big\{\tau_R^{(n)}(v)<\infty\big\} - \sum_{v\in\V_n} \1\big\{\xi_{t(n)}^v\neq \emptyset, \tau^{(n)}_R(v)\leq t(n)\big\}
        \\ &
            = \sum_{v\in\V_n}\1\big\{\tau_R^{(n)}(v)<\infty\big\} - \sum_{v\in\V_n} \1\big\{\xi_{t(n)}^v\neq \emptyset\big\}+\sum_{v\in\V_n}\1\big\{\xi_{t(n)}^v\neq \emptyset, \tau^{(n)}_R(v)> t(n)\big\}.
    \end{aligned}
\]
After dividing both sides by \(\V_n\), the right-hand side converges in probability to \(\eta_R-\eta(\mathsf Q)\), as \(n\to\infty\) by assumption and Propositions~\ref{prop:noboundedsurvival} and~\ref{prop:conversurvpath}. As \(\eta_R\downarrow \eta\), as \(R\to\infty\), the expectation in~\eqref{eq:convEBound} can be made arbitrarily small by choosing \(n\) and \(R\) large enough, proving the claimed implication.

It remains to prove the other implication, i.e.\ that~\eqref{eq:almostlocal} implies, for all $\varepsilon>0$,
\begin{equation}\label{eq:toshow43}
    \P^\lambda\big(|\xi^{\V_n}_{t(n)}|<|\V_n|(\eta-\varepsilon)\big)\to 0\qquad \text{as }n\to\infty.
\end{equation}
Using self-duality, this is equivalent to
\[\begin{aligned}
    \P^\lambda\Big(\sum_{v\in \V_n}\1\big\{\xi_{t(n)}^v\neq \emptyset\big\}<|\V_n|(\eta-\varepsilon)\Big)\to 0\qquad \text{as }n\to\infty.
\end{aligned}\]
To obtain this statement, note that
\[\begin{aligned}
    &\P^\lambda\Big(\sum_{v\in \V_n}\1\big\{\xi_{t(n)}^v\neq \emptyset\big\}<|\V_n|(\eta-\varepsilon)\Big)\le\P^\lambda\Big(\sum_{v\in \V_n}\1\big\{\xi_{t(n)}^v\neq \emptyset, \tau^{(n)}_R(v)<\infty\big\}<|\V_n|(\eta-\varepsilon)\Big)\\
    &\le\P^\lambda\Big(\sum_{v\in \V_n}\1\big\{\tau^{(n)}_R(v)<\infty\big\}<|\V_n|(\eta-\varepsilon/2)\Big)+\P^\lambda\Big(\sum_{v\in \V_n}\1\big\{\xi_{t(n)}^v= \emptyset, \tau^{(n)}_R(v)<\infty\big\}>|\V_n|\varepsilon/2\Big),
\end{aligned}\]
where the first summand tends to zero by Proposition~\ref{prop:conversurvpath}. For the second summand, we note that \(\xi^v_{t(n)}=\emptyset\) and \(\tau_R^{(n)}(v)<\infty\) together imply \(\tau_R^{(n)}(v)<t(n)\). An application of Markov's inequality then yields
\[
    \begin{aligned}
        \mathbf{P}^{\lambda}_{\G_n}\Big(\sum_{v\in \V_n}\1\big\{\xi_{t(n)}^v= \emptyset, \tau^{(n)}_R(v)<t(n)\big\}>\varepsilon|\V_n|\Big)
        &
            \le \frac{1}{\varepsilon |\V_n|}\mathbf{E}_{\G_n}^\lambda\Big[\#\big\{v\in \V_n\colon \xi_{t(n)}^v= \emptyset, \tau_R^{(n)}(v)\le t(n)\big\}\Big]
        \\&
            \leq \frac1\varepsilon\mathbf{P}^\lambda_{\G_n}\big(\xi_{t(n)}^{o_n}=\emptyset, \tau_R^{(n)}(o_n)<t(n)\big).
    \end{aligned}
\]
Taking expectations on both sides, we hence see that~\eqref{eq:almostlocal} is a sufficient criterion for~\eqref{eq:toshow43}.
\end{proof}

\subsection{Fast extinction by absence of metastability}
In this section, we prove Theorem~\ref{thm:configuration model1}. We begin with an auxiliary result.
\begin{lemma}\label{lem:sparsity}
Let $(\G_n)_{n\in\N}$ be a sequence of {finite connected sparse random graphs} with $\G_n\underset{n\to\infty}{\overset{\P}{\rightharpoonup}} (\G,o)$. Then, for every $\delta>0$, there exists $\varepsilon>0$ such that
\[
\lim_{n\to\infty}\P\Big(\max_{I\subset \V_n\colon |I|<\varepsilon |\V_n|} \sum_{v\in I}\deg_{\G_n}(v)>\delta |\V_n|\Big)=0.
\]
\end{lemma}
\begin{proof}
Fix $\delta>0$ and denote $|\V_n|=N$. Let $D_1^{(n)}\geq D_2^{(n)}\geq \dots\geq  D_N^{(n)}$ denote the vertex degrees in $\G_n$ ordered by magnitude. By sparsity and the convergence assumption, it follows that, for any $\varepsilon\in[0,1]$,
\[
    \infty>\E[ \deg_\G(o)] -\lim_{n\to\infty}N^{-1}\sum_{i=\lfloor\varepsilon N\rfloor+1}^{N}D_i^{(n)}= \lim_{n\to\infty}N^{-1}\sum_{i=1}^{\lfloor \varepsilon N\rfloor}D_i^{(n)}=:\sigma_\varepsilon,
\]
where the convergence is to be understood in probability. In particular, the deterministic term $\sigma_\varepsilon$ on the right-hand side vanishes as $\varepsilon\to 0$. Writing $\Sigma_\varepsilon(n):=N^{-1}\sum_{i=1}^{\lfloor \varepsilon N\rfloor}D_i^{(n)}$, we obtain for $\varepsilon\in(0,1)$ fixed
\[
\lim_{n\to\infty}\P(\Sigma_\varepsilon(n) > 2 \sigma_\varepsilon)=0.
\]
The desired result now follows upon choosing $\varepsilon=\varepsilon(\delta)$ such that $\sigma_{\varepsilon}<\delta/2$.
\end{proof}

{We now explain how the sparsity lemma turns absence of metastability into fast extinction on the exponential scale. Take \(\lambda<\lambda_\rho\). Then the process started from full occupancy has vanishing density at every exponential time \(\exp(c|\V_n|)\). Consequently, during an interval of length \(T=\exp(c|\V_n|)\), the process cannot spend a positive fraction of its time at positive density; if it has not yet died out, it must therefore make many separated visits to low-density states. Whenever the infected set has density at most \(\varepsilon\), Lemma~\ref{lem:sparsity} ensures that its total degree is at most \(\delta|\V_n|\), with high probability. At such a time there is an extinction attempt: all currently infected vertices recover during the next unit interval and no infection arrow leaves the infected set. The probability of this attempt is exponentially small in \(|\V_n|\), but by choosing \(\varepsilon\) and \(\delta\) small relative to \(c\), the number of low-density visits on the time scale \(T\) compensates for this cost. Thus survival up to \(T\) has vanishing probability, which is the desired inclusion \(\lambda<\lambda_+\).}

\begin{proof}[Proof of Theorem~\ref{thm:configuration model1}]
As \(\lambda_\rho\geq \lambda_1\) by Theorem~\ref{thm:metastable1}, it suffices to show \(\lambda_\rho=\lambda_+\). Clearly, \(\lambda_\rho\geq \lambda_+\) by definition, and it hence remains to show \(\lambda_\rho\leq\lambda_+\). To this end, pick $\lambda < \lambda_\rho$ and show that $\xi^{\V_n}$ does not survive on the exponential scale. Let $c>0$ be arbitrary, denote by $T=T_{n,c}=\textup{e}^{c|\V_n|}$ the relevant time scale and let
\[
r(\varepsilon,T)=T^{-1}\int_0^{T}\1\big\{\big|\xi_t^{\V_n}\big|\leq \varepsilon |\V_n|\big\}\; \d t
\]
represent the proportion of time that the infected set spends in low density states. Let
\[
	\tau_1=\inf\big\{t>0\colon \big|\xi_t^{\V_n}\big|\le \varepsilon |\V_n|\big\}
\]
and denote
\[
	\tau_k=\inf\big\{t>\tau_{k-1}+1\colon  \big|\xi^{\V_n}_t\big|\le \varepsilon |\V_n|\big\}, \quad k\geq 2.
\]
Define further $K=\max\{k\colon \tau_k\leq T-1\}$. Fix an arbitrary \(\delta>0\) and let \(\varepsilon\) be as in Lemma~\ref{lem:sparsity}. Then, at each time \(\tau_k\), \(k\leq K\), the total size of the infected set is at most \(\varepsilon |\V_n|\) and, consequently, has total degree no larger than \(\delta|\V_n|\) with probability \(1-o(1)\). Hence, conditionally on the evolution of the process up to the stopping time $\tau_k$, the probability of immediate extinction just after $\tau_k$, meaning that a recovery occurs at every vertex in the time interval $[\tau_k,\tau_k+1)$ but no infection, is at least
\[
(1-\textup{e}^{-1})^{|\xi_{\tau_k}|}\textup{e}^{-2\lambda \sum_{v\in \xi_{\tau_k}}\deg_{\G_n}(v)}\geq (1-\textup{e}^{-1})^{\varepsilon |\V_n|}\textup{e}^{-2\lambda \delta |\V_n|},
\]
where the inequality holds with probability exceeding $1-o(1)$ uniformly for all $k\leq K$. More precisely, if $E_n$ denotes the exceptional event in Lemma~\ref{lem:sparsity}, then we obtain from the strong Markov property and a coupling of the contact process to a geometric experiment,
\begin{equation*}
    \begin{aligned}
        \P^\lambda
        &
        	\big(\{\tau^{(n)}_\emptyset(\V_n)>T\} \, \big| \, |\V_n|\big)
        \\ &
        	\leq \P\big(E_n \,\big| \, |\V_n|\big)+\E^\lambda\big[\1\{E_n^{\mathsf c}\} \1\{\tau^{(n)}_\emptyset(\V_n)>T\}\1\{K>T/3\} \, \big| \, |\V_n| \big]+\P^\lambda\big(K\leq T/3 \, \big| \,|\V_n|\big)
        \\&
            \leq \P\big(E_n \, \big| \, |\V_n|\big)+\P^\lambda\big(K\leq T/3 \, \big| \,|\V_n|\big)+3T^{-1}\textup{e}^{-\log(1-\textup{e}^{-1}) \varepsilon |\V_n|+2\lambda\delta|\V_n|}.
    \end{aligned}
\end{equation*}
Here, the bound on the last term follows from the fact that $\P(G>t)\leq (pt)^{-1}$ for a $\operatorname{Geometric}(p)$ random variable $G$. Taking expectations yields
\begin{equation}\label{eq:betterbound}
\P^\lambda (\tau^{(n)}_\emptyset(\V_n)>T)\leq \P(E_n)+\P^\lambda(K\leq T/3)+\E\big[3T^{-1}\textup{e}^{-\log(1-\textup{e}^{-1}) \varepsilon |\V_n|+2\lambda\delta|\V_n|}\big].
\end{equation}
Decreasing the values of \(\varepsilon\) and \(\delta\) if needed, the last term of the right-hand side vanishes since $T=\textup{e}^{c|\V_n|}$ with fixed $c$, as \(n\to\infty\), while the first term vanishes by Lemma~\ref{lem:sparsity}. Hence it remains to show that $ \gamma_n=\P^\lambda(K\leq T/3)$ converges to $0$ as $n\to\infty$.

Note that, on the event $\{K\leq T/3 \} $, the term
$1-r(\varepsilon,T)$ is bounded from below by $2/3$. Let $W$ be uniform on $[0,T]$ under $\P^\lambda$ and independent of the graph sequence and contact process. Denote by
\[
    \rho(t)=|\V_n|^{-1}|\xi_t^{\V_n}|, t\geq 0,
\]
the density process and by $\mathcal{I}_n$ the $\sigma$-field generated by $\G_n$ and the corresponding edge marks in the network construction of $\xi$. We have
\[
	\P^\lambda(\rho(W)>\varepsilon)\geq \E^\lambda\big[\1\{K\leq  T/3\}\P(\rho(W)>\varepsilon\mid \mathcal{I}_n)\big]\geq \tfrac{2}{3}\P^\lambda(K\leq  T/3)=\tfrac{2}{3}\gamma_n.
\]
On the other hand,
\[
\P^\lambda(\rho(W)>\varepsilon)\leq \tfrac{1}{\varepsilon}\E^\lambda[\rho(W)],
\]
and by self-duality of $\xi$ we have
\[
    \E^\lambda[\rho(W)]\leq \P^\lambda(\tau^{(n)}_\emptyset(o_n)>W)\leq \P^\lambda(W\leq \sqrt T)+ \P^\lambda(\tau^{(n)}_\emptyset(o_n)>\sqrt T).
\]
The first term vanishes by choice of $W$. Note that duality implies that the second term equals
\[
    \E\big[\bfP_{\G_n}^\lambda \big(\xi_{\sqrt{T}}^{\V_n}\cap\{o_n\}\neq \emptyset\big)\big] = \E^\lambda \big[\rho(\sqrt{T})\big],
\]
as \(o_n\) is chosen uniformly. Since $\lambda<\lambda_\rho$, we have \(\rho(\sqrt{T})\to 0\) in probability, and the expectation on the right-hand side thus vanishes by dominated convergence, as \(\rho\leq 1\). Hence, $\lim_{n\to\infty}\gamma_n = 0$ and the result follows from \eqref{eq:betterbound}.
\end{proof}

\section*{Acknowledgments}
The authors used AI-assisted tools for editorial support only, such as automated retrieval of bibliography and template conversion for the LaTeX manuscript. No AI tools were used to generate mathematical results or proofs. The authors reviewed and take full responsibility for the final content.

\section*{Funding}
CM's research was funded by Deutsche Forschungsgemeinschaft (DFG, German Research Foundation) - SPP 2265 443916008. BJ and LL acknowledge the financial support of the Leibniz Association within the Leibniz Junior Research Group on \emph{Probabilistic Methods for Dynamic Communication Networks} as part of the Leibniz Competition as well as by Deutsche Forschungsgemeinschaft (DFG, German Research Foundation) under Germany's Excellence Strategy -- The Berlin Mathematics Research Center MATH+ (EXC-2046/1, EXC-2046/2, project ID: 390685689) through the project {\em EF45-3} on {\em Data Transmission in Dynamical Random Networks}.

\bibliographystyle{plainnat}
\bibliography{bib}

@misc{durrett2024contactprocessesquencheddisorder,
      title={Unusual properties of contact processes on percolated graphs}, 
      author={Rick Durrett},
      year={2025},
      eprint={2403.18592},
      archivePrefix={arXiv},
      primaryClass={math.PR},
      url={https://arxiv.org/abs/2403.18592}, 
      note ={\href{https://arxiv.org/abs/2403.18592}{\textit{arXiv:2403.18592}}},
}

@misc{barnier:hal-05064371,
      title={The contact process on Scale-Free Percolation}, 
      author={Andree Barnier and Patrick Hoscheit and Michele Salvi and Elisabeta Vergu},
      year={2025},
      eprint={2505.10582},
      archivePrefix={arXiv},
      primaryClass={math.PR},
      url={https://arxiv.org/abs/2505.10582}, 
      note = {\href{https://arxiv.org/abs/2505.10582}{\textit{arXiv:2505.10582}}},
}

@article {BLS15,
    AUTHOR = {Benjamini, Itai and Lyons, Russell and Schramm, Oded},
     TITLE = {Unimodular random trees},
   JOURNAL = {Ergod. Theory Dyn. Syst.},
  FJOURNAL = {Ergodic Theory and Dynamical Systems},
    VOLUME = {35},
      YEAR = {2015},
    NUMBER = {2},
     PAGES = {359--373},
      ISSN = {0143-3857,1469-4417},
   MRCLASS = {05C80 (05C05 60C05)},
  MRNUMBER = {3316916},
MRREVIEWER = {Anant\ P.\ Godbole},
       DOI = {10.1017/etds.2013.56},
       URL = {https://doi.org/10.1017/etds.2013.56},
}

@inproceedings {berger_spread_2005,
    AUTHOR = {Berger, Noam and Borgs, Christian and Chayes, Jennifer T. and
              Saberi, Amin},
     TITLE = {On the spread of viruses on the internet},
 BOOKTITLE = {Proceedings of the {S}ixteenth {A}nnual {ACM}-{SIAM}
              {S}ymposium on {D}iscrete {A}lgorithms},
     PAGES = {301--310},
 PUBLISHER = {ACM, New York},
      YEAR = {2005},
      ISBN = {0-89871-585-7},
   MRCLASS = {05C80 (91D30)},
}

@article {Lacker23,
    AUTHOR = {Lacker, Daniel and Ramanan, Kavita and Wu, Ruoyu},
     TITLE = {Local weak convergence for sparse networks of interacting
              processes},
   JOURNAL = {Ann. Appl. Probab.},
  FJOURNAL = {The Annals of Applied Probability},
    VOLUME = {33},
      YEAR = {2023},
    NUMBER = {2},
     PAGES = {643--688},
      ISSN = {1050-5164,2168-8737},
   MRCLASS = {60K35 (60B10 60F17 60J05 60J60 60J80 82C22)},
MRREVIEWER = {Guilherme\ Henrique\ de Paula Reis},
       DOI = {10.1214/22-aap1830},
       URL = {https://doi.org/10.1214/22-aap1830},
}

@misc{jahnel2025phasetransitionscontactprocesses,
      title={Phase transitions for contact processes on one-dimensional networks}, 
      author={Benedikt Jahnel and Lukas Lüchtrath and Christian Mönch},
      year={2025},
      eprint={2501.16858},
      archivePrefix={arXiv},
      primaryClass={math.PR},
      url={https://arxiv.org/abs/2501.16858}, 
      note = {\href{https://arxiv.org/abs/2501.16858}{\textit{arXiv:2501.16858}}},
}

@book{vdHGraphs2,
series = {Cambridge Series in Statistical and Probabilistic Mathematics, Vol. 54},
publisher = {Cambridge University Press, Cambridge},
isbn = {978-1-316-80558-9},
year = {2024},
title = {Random Graphs and Complex Networks, {V}ol.~2},
author = {van der Hofstad, Remco},
doi = {10.1017/9781316795552},
}

@article {NguyenSly25,
    AUTHOR = {Nguyen, Oanh and Sly, Allan},
     TITLE = {Subcritical epidemics on random graphs},
   JOURNAL = {Adv. Math.},
  FJOURNAL = {Advances in Mathematics},
    VOLUME = {462},
      YEAR = {2025},
     PAGES = {Paper No. 110102, 57},
      ISSN = {0001-8708,1090-2082},
   MRCLASS = {05C80 (60J27)},
       DOI = {10.1016/j.aim.2024.110102},
       URL = {https://doi.org/10.1016/j.aim.2024.110102},
}

@incollection {valesin_survey,
    AUTHOR = {Valesin, Daniel},
BOOKTITLE={Ensaios Mat.,},     
TITLE = {The contact process on random graphs},
    VOLUME = {40},
     PAGES = {1--115},
 PUBLISHER = {Soc. Brasil. Mat., Rio de Janeiro},
      YEAR = {2024},
      ISBN = {978-85-8337-240-0},
   MRCLASS = {05C80},
  MRNUMBER = {4849657},
    doi ={10.21711/217504322024/em401},
}

@article{luechtrath2024chemical,
      title={All spatial random graphs with weak long-range effects have chemical distance comparable to Euclidean distance}, 
      author={Lukas L\"{u}chtrath},
   JOURNAL = {J. Theor. Probab.},
    VOLUME = {39},
      YEAR = {2026},
    NUMBER = {12},
       doi = {10.1007/s10959-025-01467-0},
}

@article {huang2018contact,
    AUTHOR = {Huang, Xiangying and Durrett, Rick},
     TITLE = {The contact process on random graphs and {G}alton {W}atson
              trees},
   JOURNAL = {ALEA Lat. Am. J. Probab. Math. Stat.},
  FJOURNAL = {ALEA. Latin American Journal of Probability and Mathematical
              Statistics},
    VOLUME = {17},
      YEAR = {2020},
    NUMBER = {1},
     PAGES = {159--182},
      ISSN = {1980-0436},
   MRCLASS = {60K35 (05C80 60J80)},
       DOI = {10.30757/alea.v17-07},
       URL = {https://doi.org/10.30757/alea.v17-07},
}

@article{bezuidenhout_exponential_1991,
	title = {Exponential decay for subcritical contact and percolation processes},
	volume = {19},
	issn = {0091-1798,2168-894X},
	pages = {984--1009},
	number = {3},
	FJOURNAL = {The Annals of Probability},
	JOURNAL = {Ann. Probab.},
	author = {Bezuidenhout, Carol and Grimmett, Geoffrey},
	year = {1991},
	doi = {10.1214/aop/1176990332},
}

@article{durrett_contact_1982,
	title = {Contact processes in several dimensions},
	volume = {59},
	issn = {0044-3719},
	doi = {10.1007/BF00532808},
	pages = {535--552},
	number = {4},
	FJOURNAL = {Zeitschrift für Wahrscheinlichkeitstheorie und Verwandte Gebiete},
	JOURNAL = {Z. Wahrsch. Verw. Gebiete},
	author = {Durrett, Richard and Griffeath, David},
	year = {1982},
}

@article{griffeath_basic_1981,
	title = {The basic contact processes},
	volume = {11},
	issn = {0304-4149,1879-209X},
	doi = {10.1016/0304-4149(81)90002-8},
	pages = {151--185},
	number = {2},
	FJOURNAL = {Stochastic Processes and their Applications},
	JOURNAL = {Stochastic Process. Appl.},
	author = {Griffeath, David},
	year = {1981},
}

@article{griffeath_ergodic_1975,
	title = {Ergodic theorems for graph interactions},
	volume = {7},
	issn = {0001-8678,1475-6064},
	doi = {10.2307/1425859},
	pages = {179--194},
	FJOURNAL = {Advances in Applied Probability},
	JOURNAL = {Adv. Appl. Probab.},
	author = {Griffeath, David},
	year = {1975},
}

@article{gracar2023finiteness,
author = {Peter Gracar and Lukas L{\"u}chtrath and Christian M{\"o}nch},
    title = {Finiteness of the percolation threshold for inhomogeneous long-range models in one dimension},
    volume = {30},
    fjournal = {Electronic Journal of Probability},
    journal = {Electron. J. Probab.},
    publisher = {Institute of Mathematical Statistics and Bernoulli Society},
    year = {2025},
    doi = {10.1214/25-EJP1399},
}

@misc{vanderhofstad2023giant,
      title={The giant in random graphs is almost local}, 
      author={Remco van der Hofstad},
      year={2025},
      eprint={2103.11733},
      archivePrefix={arXiv},
      primaryClass={math.PR},
      url={https://arxiv.org/abs/2103.11733}, 
      note = {\href{https://arxiv.org/abs/2103.11733}{\textit{arXiv:2103.11733}}},
}

@article {aldous_processes_2018,
    AUTHOR = {Aldous, David and Lyons, Russell},
     TITLE = {Processes on unimodular random networks},
   JOURNAL = {Electron. J. Probab.},
  FJOURNAL = {Electronic Journal of Probability},
    VOLUME = {12},
      YEAR = {2007},
	number = {54},
     PAGES = {1454--1508},
      ISSN = {1083-6489},
   MRCLASS = {60C05 (05C80 60G50)},
  MRNUMBER = {2354165},
MRREVIEWER = {Jean-Fran\c cois\ Delmas},
       DOI = {10.1214/EJP.v12-463},
       URL = {https://doi.org/10.1214/EJP.v12-463},
addendum = {Extended and corrected version available at \url{https://rdlyons.pages.iu.edu/pdf/urn.pdf}}
}

@article{schapira_extinction_2017,
	title = {Extinction time for the contact process on general graphs},
	volume = {169},
	issn = {0178-8051,1432-2064},
	doi = {10.1007/s00440-016-0742-0},
	pages = {871--899},
	number = {3},
	FJOURNAL = {Probability Theory and Related Fields},
	JOURNAL = {Probab. Theory Related Fields},
	author = {Schapira, Bruno and Valesin, Daniel},
	year = {2017},
}

@article{can_super-exponential_2018,
	title = {Super-exponential extinction time of the contact process on random geometric graphs},
	volume = {27},
	issn = {0963-5483,1469-2163},
	doi = {10.1017/S0963548317000372},
	pages = {162--185},
	number = {2},
	FJOURNAL = {Combinatorics, Probability and Computing},
	JOURNAL = {Combin. Probab. Comput.},
	author = {Can, Van Hao},
	year = {2018},
}

@article{durrett_contact_1988,
	title = {The contact process on a finite set},
	volume = {16},
	issn = {0091-1798,2168-894X},
	pages = {1158--1173},
	number = {3},
	FJOURNAL = {The Annals of Probability},
	JOURNAL = {Ann. Probab.},
	author = {Durrett, Richard and Liu, Xiu Fang},
	year = {1988},
	doi = {10.1214/aop/1176991682},
}

@article{durrett_contact_1988b,
	title = {The contact process on a finite set. {II}},
	volume = {16},
	issn = {0091-1798,2168-894X},
	pages = {1570--1583},
	number = {4},
	FJOURNAL = {The Annals of Probability},
	JOURNAL = {Ann. Probab.},
	author = {Durrett, Richard and Schonmann, Roberto H.},
	year = {1988},
	doi = {10.1214/aop/1176991584},
}

@book{liggett_stochastic_1999,
	title = {Stochastic Interacting Systems: Contact, Voter and Exclusion Processes},
	volume = {324},
	isbn = {978-3-540-65995-2},
	series = {Grundlehren der mathematischen Wissenschaften [Fundamental Principles of Mathematical Sciences]},
	shorttitle = {Stochastic interacting systems},
	publisher = {Springer-Verlag, Berlin},
	author = {Liggett, Thomas M.},
	year = {1999},
	doi = {10.1007/978-3-662-03990-8},
}

@article{lalley_2017,
author = {Steven Lalley and Wei Su},
title = {{Contact processes on random regular graphs}},
volume = {27},
journal = {Ann. Appl. Probab.},
number = {4},
publisher = {Institute of Mathematical Statistics},
pages = {2061 -- 2097},
year = {2017},
doi = {10.1214/16-AAP1249},
URL = {https://doi.org/10.1214/16-AAP1249}
}

@article{linker_contact_2021,
	title = {The contact process on random hyperbolic graphs: metastability and critical exponents},
	volume = {49},
	issn = {0091-1798,2168-894X},
	doi = {10.1214/20-aop1489},
	shorttitle = {The contact process on random hyperbolic graphs},
	pages = {1480--1514},
	number = {3},
	FJOURNAL = {The Annals of Probability},
	JOURNAL = {Ann. Probab.},
	author = {Linker, Amitai and Mitsche, Dieter and Schapira, Bruno and Valesin, Daniel},
	year = {2021},
}

@article{bhamidi_survival_2021,
	title = {Survival and extinction of epidemics on random graphs with general degree},
	volume = {49},
	issn = {0091-1798,2168-894X},
	doi = {10.1214/20-AOP1451},
	pages = {244--286},
	number = {1},
	FJOURNAL = {The Annals of Probability},
	JOURNAL = {Ann. Probab.},
	author = {Bhamidi, Shankar and Nam, Danny and Nguyen, Oanh and Sly, Allan},
	year = {2021},
}

@article{cassandro_metastable_1984,
	title = {Metastable behavior of stochastic dynamics: a pathwise approach},
	volume = {35},
	issn = {0022-4715,1572-9613},
	doi = {10.1007/BF01010826},
	shorttitle = {Metastable behavior of stochastic dynamics},
	pages = {603--634},
	number = {5},
	FJOURNAL = {Journal of Statistical Physics},
	JOURNAL = {J. Statist. Phys.},
	author = {Cassandro, Marzio and Galves, Antonio and Olivieri, Enzo and Vares, Maria Eulália},
	year = {1984},
}

@book{last_lectures_2018,
	title = {Lectures on the Poisson Process},
	volume = {7},
	series = {Institute of Mathematical Statistics Textbooks},
	pagetotal = {xx+293},
	publisher = {Cambridge University Press, Cambridge},
	author = {Last, Günter and Penrose, Mathew},
	year = {2018},
	doi = {10.1017/9781316104477},
}

@article{nam_critical_2022,
	title = {Critical value asymptotics for the contact process on random graphs},
	volume = {375},
	issn = {0002-9947,1088-6850},
	doi = {10.1090/tran/8399},
	pages = {3899--3967},
	number = {6},
	FJOURNAL = {Transactions of the American Mathematical Society},
	JOURNAL = {Trans. Amer. Math. Soc.},
	author = {Nam, Danny and Nguyen, Oanh and Sly, Allan},
	year = {2022},
}

@article{bezuidenhout_critical_1990,
	title = {The critical contact process dies out},
	volume = {18},
	issn = {0091-1798,2168-894X},
	pages = {1462--1482},
	number = {4},
	FJOURNAL = {The Annals of Probability},
	JOURNAL = {Ann. Probab.},
	author = {Bezuidenhout, Carol and Grimmett, Geoffrey},
	year = {1990},
	doi = {10.1214/aop/1176990627},
}

@book{liggett_interacting_2005,
	   AUTHOR = {Liggett, Thomas M.},
     TITLE = {Interacting particle systems},
    SERIES = {Classics in Mathematics},
      NOTE = {Reprint of the 1985 original},
 PUBLISHER = {Springer-Verlag, Berlin},
      YEAR = {2005},
      ISBN = {3-540-22617-6},
   MRCLASS = {60-02 (60K35 82C22)},
       DOI = {10.1007/b138374},
       URL = {https://doi.org/10.1007/b138374},
}

@article{griffeath_binary_1983,
	title = {The binary contact path process},
	volume = {11},
	issn = {0091-1798,2168-894X},
	pages = {692--705},
	number = {3},
	fjournal = {The Annals of Probability},
	journal = {Ann. Probab.},
	author = {Griffeath, David},
	year = {1983},
	doi = {10.1214/aop/1176993514},
}

@article{harris_contact_1974,
	title = {Contact interactions on a lattice},
	volume = {2},
	issn = {0091-1798},
	doi = {10.1214/aop/1176996493},
	pages = {969--988},
	fjournal = {The Annals of Probability},
	journal = {Ann. Probab.},
	author = {Harris, T. E.},
	year = {1974},
}

\end{document}